\documentclass[twoside,a4paper,reqno,11pt]{amsart}
\usepackage{amsfonts, amsbsy, amsmath, amssymb, latexsym}
\usepackage{mathrsfs,array}
\usepackage[top=30mm,right=30mm,bottom=30mm,left=30mm]{geometry}
\usepackage{stmaryrd}
\usepackage{bm}

\usepackage{hyperref}

\newcommand{\N}{\mathbb{N}}
\newcommand{\R}{\mathbb{R}}

\renewcommand{\l}{\lambda} 

 \renewcommand{\to}{\rightarrow}

\newcommand{\leqs}{\leqslant}
\newcommand{\geqs}{\geqslant}

 \newcommand{\vs}{\vspace{3mm}}

\makeatletter
\newcommand{\imod}[1]{\allowbreak\mkern4mu({\operator@font mod}\,\,#1)}
\makeatother

\newtheorem{theorem}{Theorem}
\newtheorem*{conj*}{Conjecture}

\newtheorem{corol}[theorem]{Corollary}

\newtheorem{thm}{Theorem}[section]
\newtheorem{prop}[thm]{Proposition}
\newtheorem{lem}[thm]{Lemma}
\newtheorem{cor}[thm]{Corollary}

\theoremstyle{definition}
\newtheorem{rem}[thm]{Remark}

\newtheorem{ex}[thm]{Example}

\begin{document}

\author{Timothy C. Burness}
 \address{T.C. Burness, School of Mathematics, University of Bristol, Bristol BS8 1TW, UK}
 \email{t.burness@bristol.ac.uk}

 \author{Martin W. Liebeck}
\address{M.W. Liebeck, Department of Mathematics,
    Imperial College, London SW7 2BZ, UK}
\email{m.liebeck@imperial.ac.uk}

\author{Aner Shalev}
\address{A. Shalev, Institute of Mathematics, Hebrew University, Jerusalem 91904, Israel}
\email{shalev@math.huji.ac.il}

\title{The length and depth of algebraic groups}

\begin{abstract}
Let $G$ be a connected algebraic group. An unrefinable chain of $G$ is a chain of subgroups $G = G_0 > G_1 > \cdots > G_t = 1$, where each $G_i$ is a maximal connected subgroup of $G_{i-1}$. We introduce the notion of the length (respectively, depth) of $G$, defined as the maximal (respectively, minimal) length of such a chain. Working over an algebraically closed field, we calculate the length of a connected group $G$ in terms of the dimension of its unipotent radical
$R_u(G)$ and the dimension of a Borel subgroup $B$ of the reductive quotient $G/R_u(G)$. In particular, a simple algebraic group of rank $r$ has length $\dim B + r$, which gives a natural extension of a theorem of Solomon and Turull on finite quasisimple groups of Lie type. We then deduce that the length of any connected algebraic group $G$
exceeds $\frac{1}{2} \dim G$.

We also study the depth of simple algebraic groups. In characteristic zero, we show that the depth of such a group is at most $6$ (this bound is sharp). In the positive characteristic setting, we calculate the exact depth of each exceptional algebraic group and we prove that the depth of a classical group (over a fixed algebraically closed field of positive characteristic) tends to infinity with the rank of the group.

Finally we study the chain difference of an algebraic group, which is the difference between its length
and its depth. In particular we prove that, for any connected algebraic group $G$ with soluble radical $R(G)$, the dimension of $G/R(G)$ is
bounded above in terms of the chain difference of $G$.
\end{abstract}

\footnotetext{
The third author acknowledges the hospitality and support of Imperial College, London, while part of this work was carried out. He also acknowledges the support of ISF grant 1117/13 and the Vinik chair of mathematics which he holds.}

\date{\today}
\maketitle

\section{Introduction}\label{s:intro}

The \emph{length} of a finite group $G$, denoted by $l(G)$, is the maximum length of a chain of subgroups of $G$. This interesting invariant was the subject of several papers by Janko and Harada \cite{Harada,Jan,Jan2} in the 1960s, culminating in Harada's description of the finite simple groups of length at most $7$ in \cite{Harada}. In more recent years, the notion of length has arisen naturally in several different contexts. For example, Babai \cite{Babai} considered the length of the symmetric group $S_n$ in relation to the computational complexity of algorithms for finite permutation groups (a precise formula for $l(S_n)$ was later determined by Cameron, Solomon and Turull in \cite{CST}). Motivated by applications to fixed-point-free automorphisms of finite soluble groups, Seitz, Solomon and Turull studied the length of finite groups of Lie type in a series of papers in the early 1990s \cite{SST,ST,ST2}. Let us highlight one of their main results, \cite[Theorem A*]{ST2}, which states that if $G=G_r(p^k)$ is a finite quasisimple group of Lie type and $k$ is sufficiently large (with respect to the characteristic $p$), then
\begin{equation}\label{e:st}
l(G) = l(B)+r,
\end{equation}
where $B$ is a Borel subgroup of $G$ and $r$ is the twisted Lie rank of $G$.

The dual notion of the \emph{depth} of a finite group $G$, denoted by $\l(G)$, is the minimal length of a chain of subgroups
\[
G = G_0 > G_1 > \cdots > G_{t-1} > G_t = 1,
\]
where each $G_i$ is a maximal subgroup of $G_{i-1}$. This invariant was studied by Kohler \cite{K} for finite soluble groups and we refer the reader to more recent work of Shareshian and Woodroofe \cite{SW} for further results in the context of lattice theory. In \cite{BLS} we proved several results on the depth of finite simple groups and we studied the relationship between the length and depth of simple groups (see \cite{BLS2} for further results on the length and depth of finite groups). For instance, \cite[Theorem 1]{BLS} classifies the simple groups of depth $3$ (it is easy to see that $\l(G) \geqs 3$ for every non-abelian simple group $G$) and \cite[Theorem 2]{BLS} shows that alternating groups have bounded depth (more precisely, $\l(A_n) \leqs 23$ for all $n$, whereas $l(A_n)$ tends to infinity with $n$). Upper bounds on the depth of each simple group of Lie type over $\mathbb{F}_q$ are presented in \cite[Theorem 4]{BLS}; the bounds are given in terms of $k$, where $q=p^k$ with $p$ a prime.

In this paper, we extend the above notions of length and depth to connected algebraic groups over algebraically closed fields. Let $G$ be a connected algebraic group over an algebraically closed field of characteristic $p \geqs 0$. An \emph{unrefinable} chain of length $t$ of $G$ is a chain of subgroups
\[
G = G_0 > G_1 > \cdots > G_{t-1} > G_t = 1,
\]
where each $G_i$ is a maximal closed connected subgroup of $G_{i-1}$ (that is, $G_i$ is maximal among the proper connected subgroups of $G_{i-1}$). We define the \emph{length} of $G$, denoted by $l(G)$, to be the maximal length of an unrefinable chain.
Similarly, the \emph{depth} $\l(G)$ of $G$ is the minimal length of such a chain. Notice that we impose the condition that the subgroups in an unrefinable chain are connected, which seems to be the most natural (and interesting) definition in this setting.

In the statements of our main results, and for the remainder of the paper, we assume that the given algebraic group is connected and the underlying field is algebraically closed (unless stated otherwise). Also note that our results are independent of any choice of isogeny type. Our first result concerns the length of an algebraic group.

\begin{theorem}\label{lengthgen}
Let $G$ be an algebraic group and let $B$ be a Borel subgroup of the reductive group $\bar G = G/R_u(G)$. Then
\[
l(G) = \dim R_u(G) + \dim B + r,
\]
where $r$ is the semisimple rank of $\bar G$.
\end{theorem}

\begin{corol}\label{length}
Let $G$ be a simple algebraic group of rank $r$ and let $B$ be a Borel subgroup of $G$. Then
\[
l(G) = \dim B + r.
\]
Moreover, every unrefinable chain of $G$ of maximum length includes a maximal parabolic subgroup.
\end{corol}

The last sentence of the corollary is justified in Remark \ref{bchain}.

By Lemma \ref{sol}, the solubility of $B$ implies that $l(B) = \dim B$, so Corollary \ref{length} is the algebraic group analogue of the aforementioned result of Solomon and Turull \cite[Theorem A*]{ST2} for finite quasisimple groups (see \eqref{e:st} above).

Next, we relate the length of arbitrary algebraic groups $G$ to their dimension.
We clearly have $l(G) \leqs \dim G$.

\begin{theorem}\label{t:len}
Let $G$ be an algebraic group.
\begin{itemize}\addtolength{\itemsep}{0.2\baselineskip}
\item[{\rm (i)}] $l(G) > \frac{1}{2}\dim G$.
\item[{\rm (ii)}] $l(G) = \dim G$ if and only if $G/R(G) \cong A_1^t$ for some $t \geqs 0$,
where $R(G)$ is the soluble radical of $G$.
\end{itemize}
\end{theorem}

The lower bound in part (i) of Theorem \ref{t:len} is essentially best possible. For example, if $G = C_r$ is a symplectic group of rank $r \geqs 1$, then Corollary \ref{length} implies that
\[
\frac{l(G)}{\dim G} = \frac{1}{2} + \frac{3}{4r+2} \to \frac{1}{2} \mbox{ as $r \to \infty$.}
\]

We now turn to the depth of simple algebraic groups. Our first result shows that simple algebraic groups in characteristic zero have bounded depth.

\begin{theorem}\label{alg0}
Let $G$ be a simple algebraic group in characteristic zero. Then
\[
\l(G) = \left\{\begin{array}{ll}
3 & \mbox{if $G = A_1$} \\
5 & \mbox{if $G = A_r$ {\rm ($r \geqs 3$, $r \ne 6$)}, $B_3$, $D_r$ or $E_6$} \\
6 & \mbox{if $G=A_6$} \\
4 & \mbox{in all other cases.}
\end{array}\right.
\]
\end{theorem}

Our next result shows that the depth of simple groups in the positive characteristic setting is rather different. In particular, the depth can be arbitrarily large. To state this result, we need some additional notation. Given a prime $p$, define a sequence $e_n(p)$ ($n\in \N$) as follows: $e_1(p)=p$, and for $l>1$,
\[
e_{l+1}(p) = p^{e_l(p)^2}.
\]
Now define a function $\psi_p : \R\to \N$ by
\[
\psi_p(x) = {\rm  min}\left(l \, : \, e_l(p) \geqs x\right).
\]

\begin{theorem}\label{algp}
Let $G$ be a simple algebraic group in characteristic $p>0$ with rank $r$.
\begin{itemize}\addtolength{\itemsep}{0.2\baselineskip}
\item[{\rm (i)}] If $G$ is an exceptional group then $\l(G) \leqs 9$, with equality if and only if $G=E_8$ and $p=2$.
\item[{\rm (ii)}] If $G$ is a classical group, then $\l(G) \leqs 2(\log_2r)^2+12$.
\item[{\rm (iii)}] For any $G$, we have $\l(G) \geqs \psi_p(r)$. In particular, $\l(G)\to \infty$ as $r\to \infty$.
\end{itemize}
\end{theorem}

Part (i) of Theorem \ref{algp} is an immediate corollary of Theorem \ref{algp:ex}, which gives the exact depth of each exceptional algebraic group. For parts (ii) and (iii), see Theorems \ref{clup} and Theorem \ref{t:lim}, respectively. We also give an example (Example \ref{slow}) to show that the lower bound $\psi_p(r)$ in (iii) is of roughly the correct order of magnitude.

By a well-known theorem of Iwasawa \cite{I}, the length and depth of a finite group $G$ are equal if and only if $G$ is supersoluble. This result does not extend directly to algebraic groups. However, it follows from our results on length and depth that the only simple algebraic group with $\l(G)=l(G)$ is $G = A_1$ (see Lemma \ref{cd0simple}). More generally, we prove the following result on arbitrary connected groups with this property, which can be viewed as a partial analogue of Iwasawa's theorem.

\begin{theorem}\label{cd0}
Let $G$ be an algebraic group satisfying $\l(G)=l(G)$ and let $R(G)$ be the radical of $G$. Then either $G$ is soluble, or $G/R(G) \cong A_1$.
\end{theorem}

In fact it follows from our arguments that $\l(G) = l(G)$ if and only if $\l(G) = \dim G$.

This is proved in Section \ref{cor3}. Note that the converse is false: for example, if $G=UA_1$, a semidirect product where $U$ is a nontrivial irreducible $KA_1$-module, then $A_1$ is maximal in $G$, so $\l(G)\leqs 1+\l(A_1)=4$, while $l(G) = \dim U+l(A_1)>4$. On the other hand, if $G = U\times A_1$, or if $G$ is a nonsplit extension of the irreducible module $U$ by $A_1$, then $\l(G)=l(G)$.

More generally, we can consider the \emph{chain difference} of $G$, which is defined by
\[
{\rm cd}(G) = l(G) - \l(G).
\]
This invariant was studied for finite simple groups. See \cite{BWZ, HS, pet} for the study of finite
simple groups of chain difference one, and Corollary 9 in \cite{BLS}, where we bound the length
of a finite simple group in terms of its chain difference. For algebraic groups we prove a
stronger result, without assuming simplicity.

\begin{theorem}\label{t:cd}
Let $G$ be an algebraic group and set $\bar{G} = G/R(G)$. Then
\[
\dim \bar{G} \leqs  \left(2 + o(1)\right){\rm cd}(G),
\]
where $o(1) = o_{{\rm cd}(G)}(1)$. More precisely,
\[
\dim \bar{G} \leqs 2\,{\rm cd}(G) + 40\sqrt{400+2{\rm cd}(G)}+800.
\]
\end{theorem}

This result will be proved in Section \ref{s:cd}.

We also consider the {\it chain ratio} ${\rm cr}(G) = \l(G)/l(G)$ of an algebraic group $G$, and show in Section \ref{chrat} that in contrast to the chain difference, the dimension of $G/R(G)$ is not in general bounded in terms of ${\rm cr}(G)$.

\section{Preliminaries}\label{s:prel}

As stated in Section \ref{s:intro}, for the remainder of the paper we assume $G$ is a connected algebraic group over an algebraically closed field (unless stated otherwise). We start with the following elementary observation.

\begin{lem}\label{add}
Let $G$ be an algebraic group and let $N$ be a connected normal subgroup.
\begin{itemize}\addtolength{\itemsep}{0.2\baselineskip}
\item[{\rm (i)}] $\l(G) \leqs l(G) \leqs \dim G$.
\item[{\rm (ii)}] $l(G) = l(N)+l(G/N)$.
\item[{\rm (iii)}] $\l(G/N) \leqs \l(G) \leqs \l(N) + \l(G/N)$.
\end{itemize}
\end{lem}

\begin{proof}
Parts (i) and (iii) are obvious, and part (ii) is proved just as \cite[Lemma 2.1]{CST}.
\end{proof}

Recall that if $U$ is a connected unipotent algebraic group, then the \emph{Frattini subgroup} $\Phi(U)$ of $U$ is the intersection of the closed  subgroups of $U$ of codimension 1 (see \cite{fau}).

\begin{lem}\label{sol}
If $G$ is a soluble algebraic group, then $\l(G)= l(G) = \dim G$.
\end{lem}

\begin{proof}
It is sufficient to show that any connected maximal subgroup $M$ of $G$ has codimension $1$. Write $G = UT$, where $U = R_u(G)$ and $T$ is a maximal torus. If $U\leqs M$ then $M=US$, where $S$ is a connected maximal subgroup of $T$, and the result follows since $\dim S = \dim T - 1$. Now assume $U \not \leqs M$, so $M = (M \cap U)T$ and $M \cap U$ is a maximal  $T$-invariant subgroup of $U$. Now $\Phi(U)\leqs M$, so by factoring out $\Phi(U)$ we can assume that $\Phi(U)=1$. Then $U \cong K^n$, an $n$-dimensional vector space over the underlying algebraically closed field $K$ (see \cite[Proposition 1]{fau}). Moreover, $T$ acts linearly on $U$, and since $T$ is diagonalisable on $V$, a maximal $T$-invariant subspace has codimension 1 (this is proved in much greater generality in \cite[Theorem B]{mc}). Hence $M$ has codimension 1 in $G$ in this case also.
\end{proof}

\begin{lem}\label{depth23}
Let $G$ be an algebraic group and let $m \in \{1,2,3\}$. Then $\l(G) = m$ if and only if $\dim G = m$.
\end{lem}

\begin{proof}
The case $m=1$ is obvious, so assume $m \in \{2,3\}$.
If $\l(G)=2$ then $G$ has a maximal $T_1$ or $U_1$ subgroup and clearly $G$ is soluble, so $\dim G = 2$ by Lemma \ref{sol}. Conversely, if $\dim G = 2$ then $G$ is soluble and once again the result follows from Lemma \ref{sol}.

Now assume $\dim G = 3$. By the $m=2$ case already proved, $\l(G) \geqs 3$. If $G$ is soluble, then Lemma \ref{sol} implies that $\l(G)=3$, otherwise $G=A_1$ and $\l(G)=3$ since
\[
A_1 > U_1T_1 > T_1 > 1
\]
is unrefinable (here we write $U_k$ for a unipotent group of dimension $k$, and similarly $T_k$ for a $k$-dimensional torus). Finally, suppose $\l(G)=3$ and assume that $G$ is insoluble. If $G$ is reductive, it has a maximal connected subgroup of depth 2, hence of dimension 2, and the only possibility is $G = A_1$. Otherwise, let $U = R_u(G)$ be the unipotent radical of $G$. Since $G$ is insoluble, the previous sentence implies that $G/U \cong A_1$. But $G$ has a maximal subgroup of depth 2, which must be soluble of dimension 2. This is clearly not possible.
\end{proof}

\begin{rem}
Notice that the conclusion of Lemma \ref{depth23} does not extend to integers $m>3$. For example, if $r \geqs 2$ then the symplectic group $C_r$ has a maximal $A_1$ subgroup in characteristic $0$, so there are depth $4$ algebraic groups of arbitrarily large dimension.
\end{rem}

\begin{lem}\label{l:ss}
Let $G = UL$ be an algebraic group, where $U$ and $L$ are nontrivial connected subgroups of $G$, with $U$ normal. Then $\l(G) \geqs 1+\l(L)$.
\end{lem}

\begin{proof}
By Lemma \ref{add}(iii), $\l(G)\geqs \l(L)$. Assume for a contradiction that $\l(G)=\l(L)=t$, and let
\begin{equation}\label{unref}
G = G_t > G_{t-1} > \cdots > G_1 > G_0 = 1
\end{equation}
be an unrefinable chain. If $L_i = G_iU/U\leqs L$, then we must have $L_i < L_{i+1}$ for all $i$, since otherwise the depth of $L$ would be less than $t$. This means that
\[
G = G_tU > G_{t-1}U > \cdots > G_1U > G_0U = U
\]
is a chain of connected subgroups. But then if we choose $i$ minimal such that $U \not \leqs G_i$, we have $G_i < G_iU < G_{i+1}$, contradicting the unrefinability of \eqref{unref}.
\end{proof}

\section{Proofs}\label{s:proofs}

\subsection{Proof of Theorem \ref{lengthgen}}

Let $G$ be an algebraic group. The proof goes by induction on $\dim G$. For $\dim G = 1$, the result is obvious.

Write $U = R_u(G)$. By Lemmas \ref{add}(ii) and \ref{sol},
\[
l(G) = l(U)+l(G/U) = \dim U + l(G/U).
\]
If $U \ne 1$ then the conclusion follows by induction, so we assume that $U=1$; that is, $G$ is reductive. Write $G = G_1\cdots G_tZ$, a commuting product with each $G_i$ simple and $Z=Z(G)^0$, and let $B_i$ be a Borel subgroup of $G_i$. Note that $B = B_1\cdots B_tZ$ is a Borel subgroup of $G$. Let $r_i$ be the rank of $G_i$, so $r=\sum_{i}r_i$ is the semisimple rank of $G$.  By Lemma \ref{add}(ii) we have
\[
l(G) = \sum_{i} l(G_i)+\dim Z.
\]
If $t>1$ or $Z\ne 1$ we can apply induction to deduce that $l(G_i) = l(B_i)+r_i$ for each $i$, and hence
\[
l(G) = \sum_i \dim B_i + \sum_i r_i + \dim Z = \dim B + r,
\]
as required. Hence we may assume that $G$ is simple of rank $r$. By considering an unrefinable chain passing through $B$, noting that $l(B) = \dim B$ by Lemma \ref{add}, it follows that
\begin{equation}\label{e:borel}
l(G) \geqs \dim B + r.
\end{equation}
Our goal is to show that equality holds.

Let $M$ be a maximal connected subgroup of $G$ with $l(M) = l(G)-1$. By \cite[Corollary 3.9]{BT}, $M$ is either parabolic or reductive. Suppose first that $M$ is reductive and let $B_M$ be a Borel subgroup of $M$. By induction,
\[
l(M) = \dim B_M + {\rm rank}(M').
\]
But it is easy to see that $\dim B_M<\dim B-1$ and thus $l(G) < \dim B + r$. This contradicts \eqref{e:borel}, so we have reduced to the case where $M$ is a maximal parabolic subgroup.

Write $M = QL$ where $Q=R_u(M)$ and $L$ is a Levi subgroup. By induction,
\[
l(M) = \dim Q + \dim B_L + {\rm rank}(L'),
\]
where $B_L$ is a Borel subgroup of $L$. Since $B = QB_L$ is a Borel subgroup of $G$, and ${\rm rank}(L') = r-1$, it follows that $l(G) = l(M)+1 = \dim B+ r$, as required. This completes the proof of Theorem \ref{lengthgen}.

\begin{rem}\label{bchain}
Let $G$ be a simple algebraic group. By the proof of Theorem \ref{lengthgen}, it follows that every unrefinable chain of $G$ of maximum length includes a maximal parabolic subgroup. This gives  Corollary \ref{length}.
\end{rem}

\subsection{Proof of Theorem \ref{t:len}}

First consider (i). In view of Lemma \ref{sol}, the bound is clear if $G$ is soluble, so assume otherwise. Write $G/R(G) = G_1 \cdots G_t$, where each $G_i$ is simple. By applying Corollary \ref{length}, it is easy to see that $l(G_i) > \frac{1}{2}\dim G_i$ for each $i$, so by combining Lemmas \ref{add}(ii) and \ref{sol} we get
\[
l(G) = l(R(G)) + \sum_{i}l(G_i) > \dim R(G) + \frac{1}{2}\sum_{i}\dim G_i \geqs \frac{1}{2}\dim G
\]
as required. An entirely similar argument establishes (ii), noting that by Corollary \ref{length}, $l(G_i) = \dim G_i$ if and only if $G_i=A_1$ (this is easily deduced from Corollary \ref{length}; see Lemma \ref{cd0simple}).

\subsection{Proof of Theorem \ref{alg0}}

Let $G$ be a simple algebraic group over an algebraically closed field of characteristic $0$. The maximal connected subgroups of $G$ were determined by Dynkin \cite{Dynkin2, Dynkin} and we repeatedly apply these results throughout the proof. To begin with, let us assume $G$ is a classical group of rank $r$.

By Lemma \ref{depth23}, $\l(G) \geqs 3$ with equality if and only if $G=A_1$, so we may assume $r \geqs 2$. As observed by Dynkin \cite{Dynkin}, $G=C_r$ has an irreducible maximal subgroup of type $A_1$ and thus $\l(G)=4$. Similarly, $\l(G)=4$ if $G = B_r$ and $r \ne 3$. The group $G=B_3$ needs special attention because it does not have a maximal $A_1$ subgroup (indeed, an irreducibly embedded $A_1$ is contained in a $G_2$ subgroup of $G$). One checks that $\dim M > 3$ for every maximal connected subgroup $M$ of $G$, so Lemma \ref{depth23} implies that $\l(G) \geqs 5$. In fact, we see that equality holds since
\[
B_3>G_2>A_1>U_1T_1>T_1>1
\]
is an unrefinable chain of length $5$.

Next assume $G=D_r$, so $r \geqs 3$ by simplicity. Here $\l(G) \geqs 5$ since $G$ does not have a maximal $A_1$ subgroup. But $G$ does have an unrefinable chain of length $5$:
\[
\left\{\begin{array}{l}
D_r > B_{r-1} > A_1 > U_1T_1>T_1>1\;\; \mbox{if $r \ne 4$} \\
D_4 > A_2 > A_1 > U_1T_1>T_1>1
\end{array}\right.
\]
and thus $\l(G)=5$. To complete the proof for classical groups, suppose $G=A_r$ and $r \geqs 2$. Here $G$ has a maximal $A_1$ subgroup if and only if $r=2$, so we get $\l(G)=4$ if $r=2$, otherwise $\l(G) \geqs 5$. It is easy to see that $\l(G)=5$ if $r \geqs 3$ and $r \ne 6$. Indeed, we have the following unrefinable chains:
\[
\left\{\begin{array}{ll}
A_r > B_{r/2} > A_1 > U_1T_1>T_1>1 &  \mbox{if $r \geqs 4$ even, $r \ne 6$} \\
A_r > C_{(r+1)/2} > A_1 > U_1T_1>T_1>1 & \mbox{if $r \geqs 3$ odd.}
\end{array}\right.
\]
Note that if $r=6$ then the first chain is refinable (as noted above, $A_1$ is non-maximal in $B_3$) and we get $\l(G) \leqs 6$ via
\[
A_6 > B_3 > G_2 > A_1 > U_1T_1>T_1>1.
\]
We claim that $\l(G)=6$ in this case. To see this, let $M$ be a maximal connected subgroup of $G$. By \cite{Dynkin}, either $M=B_3$ or $M$ is a parabolic subgroup of the form $U_6A_5T_1$, $U_{10}A_4A_1T_1$ or $U_{12}A_3A_2T_1$ (here $U_k$ denotes a normal unipotent subgroup of dimension $k$). If $M$ is parabolic, then  Lemma \ref{add}(iii) implies that $\l(M) \geqs \l(A_k)$ for some $k \in \{3,4,5\}$ and thus $\l(M) \geqs 5$ by our above work. Since $\l(B_3)=5$, the claim follows.

Finally, let us assume $G$ is an exceptional group. By \cite{Dynkin2}, $G$ has a maximal $A_1$ subgroup if and only if $G \ne E_6$, so $\l(G)=4$ in these cases. For $G=E_6$ we have
$\l(G) \geqs 5$ and equality holds since $G$ has a maximal $G_2$ subgroup (see \cite{Dynkin2}) and so there is an unrefinable chain
\[
E_6 > G_2 > A_1 > U_1T_1>T_1>1.
\]
This completes the proof of Theorem \ref{alg0}.

\subsection{Proof of Theorem \ref{algp}(i)}\label{low}

Let $G$ be a simple algebraic group of rank $r$ over an algebraically closed field $K$ of characteristic $p > 0$. In this subsection we determine the precise depth of $G$ in the case where $G$ is of exceptional type.

We start with a preliminary lemma, which gives the precise depth of the simple algebraic groups of rank at most $4$. In Table \ref{tab:cl1}, if the final entry $c$ in a column occurs in the row corresponding to $p=\ell$, then $\l(G) = c$ for all $p \geqs \ell$. For example, Table \ref{tab:cl1} indicates that
\[
\l(C_4) = \left\{\begin{array}{ll}
7 & \mbox{if $p=2$} \\
6 & \mbox{if $p \in \{3,5,7\}$} \\
4 & \mbox{if $p \geqs 11$.}
\end{array}\right.
\]

\begin{lem}\label{l:lowrank}
Let $G$ be a simple algebraic group of rank $r \leqs 4$ in characteristic $p > 0$. Then $\l(G)$ is given in Table \ref{tab:cl1}.
\end{lem}

\begin{table}
\[
\begin{array}{ccccccccccccc} \hline
p & A_1 & A_2 & B_2 & G_2 & A_3 & B_3 & C_3 & A_4 & B_4 & C_4 & D_4 & F_4 \\ \hline
2 & 3 & 6 & 5 & 5 & 6 & 6 & 6 &  9  & 7 & 7 & 7 & 8 \\
3 &    & 4 & 5 & 5 & 5 & 6 & 5 &  6  & 5 & 6 & 7 & 6 \\
5 &    &    & 4 & 5 &    & 6 & 5 &  5  & 5 & 6 & 5 & 6 \\
7 &    &    &     & 4 &    & 5 & 4 &      & 5 & 6 &    & 5 \\
11 &  &    &      &    &    &    &    &      & 4 & 4 &    & 5 \\
> 11&&    &      &    &    &    &    &      &    &    &    & 4 \\ \hline
\end{array}
\]
\caption{The depth of low rank simple algebraic groups}
\label{tab:cl1}
\end{table}

\begin{proof}
First recall that $\l(G) \geqs 3$, with equality if and only if $G = A_1$. Now assume $r \geqs 2$ and let
\[
G > M > M_1 > \cdots > M_t = 1
\]
be an unrefinable chain of minimal length. Recall that $M$ is either parabolic or reductive. If $G$ is an exceptional group, then the possibilities for $M$ have been determined by Liebeck and Seitz (see \cite[Corollary 2(ii)]{LS}). Similarly, if $G$ is a classical group with natural module $V$, then \cite[Theorem 1]{LS98} implies that either $M$ stabilises a proper nontrivial subspace of $V$, or a tensor product decomposition of the form $V = U \otimes W$, or $M$ is simple and $V|_{M}$ is an irreducible $KM$-module with $p$-restricted highest weight. We refer the reader to \cite[Table 5]{Thomas} for a convenient list of the relevant reductive maximal connected subgroups of $G$. It will be useful to observe that $\l(M) \geqs 5$ if $M$ is a maximal parabolic subgroup of $G$ (this follows immediately from Lemma \ref{l:ss}).

Suppose $G = A_2$. If $p \geqs 3$ then $G$ has a maximal $A_1$ subgroup, so $\l(G)=4$. On the other hand, if $p=2$ then $M = U_2A_1T_1$ and $M_1 \in \{A_1T_1, U_2A_1\}$ are the only possibilities, so $\l(G) = 2+\l(M_1) = 6$.

Next let $G = B_2$ or $C_2$. If $p \geqs 5$ then $A_1$ is a maximal subgroup and thus $\l(G)=4$. If $p \in \{2,3\}$ then $\dim M > 3$ and thus $\l(G) \geqs 5$. In fact, equality holds since there is a chain
\[
B_2 > A_1A_1 > A_1 > U_1T_1 > T_1 > 1,
\]
where $A_1 < A_1A_1$ is diagonally embedded. Next assume $G = G_2$. In the usual manner, we deduce that $\l(G)=4$ if $p \geqs 7$, so let us assume $p \in \{2,3,5\}$. Here $\dim M > 3$ and thus $\l(G) \geqs 5$. Since
\[
G_2 > A_1\tilde{A}_1 > A_1 > U_1T_1 > T_1 > 1
\]
is unrefinable, we conclude that $\l(G)=5$.

Now assume $G = A_3$. Here $\dim M > 3$, so $\l(G) \geqs 5$. If $p \geqs 3$ then
\[
A_3 > A_1A_1 > A_1 > U_1T_1 > T_1 > 1
\]
is unrefinable and thus $\l(G)=5$. Suppose $p=2$, so either $M=B_2$ or $M$ is parabolic. Since $\l(B_2) = 5$ as above, and $\l(M) \geqs 5$ when $M$ is parabolic, it follows that $\l(G) = 6$.

Next suppose $G = B_3$, so $\l(G) \geqs 5$ since $\dim M > 3$. Now $G$ has a maximal  $G_2$ subgroup, so $\l(G) \leqs \l(G_2)+1$ and thus $\l(G)=5$ if $p \geqs 7$, otherwise $\l(G) \leqs 6$. By considering the various possibilities for $M$, it is easy to check that $\l(M) \geqs 5$ when $p \in \{2,3,5\}$ (we can assume $M$ is reductive, so the possibilities are listed in \cite[Table 5]{Thomas}) and we conclude that $\l(G)=6$. The case $G = C_3$ is similar. If $p \geqs 7$ then $M = A_1$ and $\l(G)=4$. Now assume $p \in \{2,3,5\}$, so $\l(G) \geqs 5$. If $p \in \{3,5\}$ then we can take $M = A_1A_1$, which gives $\l(G)=5$. Finally, if $p=2$ then one checks that $\l(M) \geqs 5$, with equality if $M = G_2$, hence $\l(G)=6$.

Suppose $G=A_4$. Here $\l(G) \geqs 5$ since $\dim M > 3$. If $p \geqs 5$ then
\[
A_4 > B_2 > A_1 > U_1T_1 > T_1 > 1
\]
is unrefinable and thus $\l(G)=5$. If $p=3$ then either $M = B_2$ or $M$ is a parabolic subgroup, whence $\l(G)=6$ since $\l(B_2)=5$ as above. Finally, suppose $p=2$. Here every maximal connected subgroup of $G$ is parabolic, so we need to consider the depth of $P_1 = U_4A_3T_1$ and $P_2 = U_6A_2A_1T_1$. The Levi factor of $P_1$ is maximal, so $P_1 > A_3T_1 > A_3$ is unrefinable and thus $\lambda(P_1) \leqs 8$ since $\lambda(A_3) = 6$. Now Lemma \ref{l:ss} gives  
$\lambda(P_1) \geqs 2+\l(A_3)=8$ and $\lambda(P_2) \geqs 3+\lambda(A_2) = 9$, so 
$\lambda(P_1) = 8$ and $\l(G)=9$.

Next assume $G=D_4$. Once again, $\l(G) \geqs 5$. If $p \geqs 5$ then $\l(G)=5$ since
\[
D_4 > A_2 > A_1 > U_1T_1 > T_1 > 1
\]
is unrefinable. Now assume $p \in \{2,3\}$. Here $\l(G) \leqs \l(B_3)+1 = 7$ and we claim that equality holds. To see this, first observe that $M$ is either a parabolic subgroup, or $M=B_3$, $A_1B_2$ ($p=3$), $A_1^4$ or $A_2$ ($p=2$). Note that $\l(B_3)=6$ and $\l(A_2)=6$ (with $p=2$ in the latter case). By applying Lemma \ref{l:ss}, it is also easy to see that $\l(M) \geqs 6$ in the remaining cases. For example, if $M = U_6A_3T_1$ then $\l(M) \geqs 1+\l(A_3T_1) \geqs 7$. This justifies the claim.

Next consider $G=B_4$. First observe that $\l(G)=4$ if and only if $p \geqs 11$. If $p \in \{3,5,7\}$ then we can take $M=A_1A_1$, which gives $\l(G)=5$. Now assume $p=2$. We claim that $\l(G)=7$. Certainly, $\l(G) \leqs 7$ since there is a chain
\[
B_4 > B_2B_2 > B_2 > A_1A_1 > A_1 > U_1T_1 > T_1 > 1.
\]
To establish equality, we need to consider the possibilities for $M$. If $M$ is reductive, then $M = D_4$, $A_1B_3$ or $B_2B_2$. By our earlier work, $\l(D_4)=7$ and $\l(A_1B_3) \geqs 1+\l(B_3) =7$. Similarly, $\l(B_2B_2) =6$ and one checks that $\l(M) \geqs 6$ if $M$ is parabolic. The claim follows.

Now assume $G=C_4$. As in the previous case, $\l(G)=4$ if and only if $p \geqs 11$. If $p \in \{3,5,7\}$ then
\[
C_4 > A_1^3 > A_1A_1 > A_1 > U_1T_1 > T_1 > 1
\]
is unrefinable and thus $\l(G) \leqs 6$. In fact, it is easy to see that $\l(M) \geqs 5$ for every connected maximal subgroup $M$ of $G$ and thus $\l(G)=6$.
Finally, suppose $p=2$. Here $\l(G) \leqs 7$ via the chain
\[
C_4 > C_2C_2 > C_2 > A_1A_1 > A_1 > U_1T_1 > T_1 > 1.
\]
We claim that $\l(G)=7$. To see this, we need to show that $\l(M) \geqs 6$ for every connected maximal subgroup $M$ of $G$. If $M$ is reductive, then $M = C_2C_2$, $A_1C_3$ or $D_4$. By combining Lemma \ref{l:ss} with our earlier work, we have $\l(A_1C_3) \geqs 1+\l(C_3)=7$ and $\l(D_4)=7$. It is also easy to see that $\l(C_2C_2) = 6$. It is routine to verify the claim when $M$ is parabolic. For instance, if $M = U_{10}A_3T_1$ then $\l(M) \geqs 1+\l(A_3T_1) \geqs 2+\l(A_3)= 8$.

To complete the proof of the lemma, we may assume $G=F_4$. Here $G$ has a maximal $A_1$ subgroup if and only if $p \geqs 13$, so we may assume $p<13$. If $p \in \{7,11\}$ then $\l(G)=5$ via the chains
\[
\left\{\begin{array}{ll}
F_4 > B_4 > A_1 > U_1T_1 > T_1 > 1 & \mbox{if $p=11$} \\
F_4 > G_2 > A_1 > U_1T_1 > T_1 > 1 & \mbox{if $p=7$.}
\end{array}\right.
\]
Next assume $p \in \{3,5\}$. Here $\l(G) \leqs 6$ since there is a chain
\[
F_4 > A_2\tilde{A}_2 > A_2 > A_1 > U_1T_1> T_1 > 1.
\]
By considering the various possibilities for $M$ and using Lemma \ref{l:ss} and our earlier work, it is easy to show that $\l(M) \geqs 5$ and thus $\l(G)=6$. Finally, let us assume $p = 2$. First observe that $\l(G) \leqs 8$ via the chain
\[
F_4 > C_4 > C_2C_2 > C_2 > A_1A_1 > A_1 > U_1T_1 > T_1 > 1.
\]
We claim that $\l(G)=8$. To see this, we need to show that $\l(M) \geqs 7$ for every maximal connected subgroup $M$ of $G$. If $M$ is reductive, then $M= C_4$, $B_4$ or $A_2\tilde{A}_2$. As above, we have $\l(B_4) = \l(C_4) = 7$ and $\l(A_2\tilde{A}_2) \geqs 1 + \l(A_2) = 7$ and the result follows. If $M = UL$ is a parabolic subgroup, with unipotent radical $U$, then Lemma \ref{l:ss} gives $\l(M) \geqs 1+\l(L)$ and it is straightforward to see that $\l(L) \geqs 6$. For example, if $M = U_{20}A_1A_2T_1$ then Lemma \ref{l:ss} yields $\l(L) \geqs 2+\l(A_2) = 8$. The result follows.
\end{proof}

We are now in a position to prove our main result for exceptional groups. In particular, part (i) of Theorem \ref{algp} is an immediate corollary of the following result. In Table \ref{tab:ex}, we adopt the same conventions as in Table \ref{tab:cl1}.

\begin{thm}\label{algp:ex}
Let $G$ be an exceptional algebraic group in characteristic $p > 0$. Then $\l(G)$ is given in Table \ref{tab:ex}.
\end{thm}

\begin{table}
\[
\begin{array}{cccccc} \hline
p & G_2 & F_4 & E_6 & E_7 & E_8 \\ \hline
2 & 5 & 8 & 6 & 8 & 9 \\
3 & 5 & 6 & 6 & 7 & 7 \\
5 & 5 & 6 & 5 & 5 & 5 \\
7  & 4 & 5 &  & 5 & 5 \\
11 &  & 5 & &  5 &  5 \\
13 &   & 4 & & 5 & 5 \\
17 &  &  & & 4 & 5 \\
19 &   & & & & 5 \\
>19 &  &  &  &  & 4 \\ \hline
\end{array}
\]
\caption{The depth of exceptional algebraic groups}
\label{tab:ex}
\end{table}

\begin{proof}
In view of Lemma \ref{l:lowrank}, we may assume $G = E_6$, $E_7$ or $E_8$. Let
\[
G > M > M_1 > \cdots > M_t = 1
\]
be an unrefinable chain of minimal length. Recall that the possibilities for $M$ are determined in \cite{LS}.

First assume $G = E_6$ and note that $\l(G) \geqs 5$ since $\dim M > 3$. If $p \geqs 5$, then $G$ has a maximal $A_2$ subgroup and $\l(A_2) = 4$ by Lemma \ref{l:lowrank}, so $\l(G)=5$.
Now assume $p \in \{2,3\}$. Here $G_2<G$ is maximal and thus $\l(G) \leqs \l(G_2)+1 = 6$.
We claim that $\l(G)=6$. If $M$ is reductive, then \cite{LS} implies that
\[
M \in \{A_1A_5, A_2^3, G_2, C_4 \, (p=3), F_4, A_2G_2\}
\]
and it is easy to check that $\l(M) \geqs 5$. By applying Lemma \ref{l:ss}, we see that the same conclusion holds when $M$ is parabolic. This justifies the claim.

Next assume $G = E_7$. Here $\l(G)=4$ if and only if $p \geqs 17$. If $5 \leqs p \leqs 13$ then
\[
E_7>A_1A_1> A_1 > U_1T_1 > T_1 > 1
\]
is unrefinable and thus $\l(G)=5$. Now assume $p=3$. First observe that $\l(G) \leqs 7$ via the chain
\[
E_7 > A_1G_2 > A_1A_2 > A_1A_1 > A_1 > U_1T_1 > T_1 > 1.
\]
We claim that $\l(G)=7$. To see this, let $M$ be a maximal connected subgroup of $G$. Suppose $M$ is reductive, in which case
\[
M \in \{ A_1D_6, A_7, A_2A_5, A_1G_2, A_1F_4, G_2C_3 \}.
\]
Now $\l(D_6) \geqs 5$ and $\l(A_5) \geqs 5$ (neither group has a maximal $A_1$ subgroup) and one can readily check that $\l(A_7) \geqs 6$ (the only reductive maximal connected subgroups are $C_4$, $D_4$ and $A_1C_3$). Therefore, by applying Lemmas \ref{l:ss} and \ref{l:lowrank} we deduce that $\l(M) \geqs 6$. Similarly, one checks that the same bound holds if $M$ is parabolic and the claim follows.

Now assume $G=E_7$ and $p=2$. Here there is an unrefinable chain
\[
E_7 > G_2C_3 > G_2G_2 > G_2 > A_1\tilde{A}_1 > A_1 > U_1T_1 > T_1 > 1
\]
and thus $\l(G) \leqs 8$. By essentially repeating the above argument for $p=3$, it is straightforward to show that $\l(M) \geqs 7$ for every maximal connected subgroup $M$ of $G$, whence $\l(G)=8$.

To complete the proof of the theorem, we may assume $G = E_8$. If $p \geqs 23$ then $\l(G)=4$ since $G$ has a maximal $A_1$ subgroup. If $5 \leqs p  \leqs 19$ then $\l(G)=5$ via the chain
\[
E_8>B_2> A_1 > U_1T_1 > T_1 > 1.
\]
Now assume $p=3$. Here $\l(G) \leqs 7$ since there is a chain
\[
E_8 > A_8 > A_2A_2 > A_2 > A_1 > U_1T_1 > T_1 > 1.
\]
To see that $\l(G)=7$, we need to show that $\l(M) \leqs 6$. If $M$ is reductive then \cite[Corollary 2(ii)]{LS} implies that
\begin{equation}\label{e:M}
M \in \{D_8, A_1E_7, A_8, A_2E_6, A_4A_4, G_2F_4\}
\end{equation}
and it is straightforward to show that $\l(M) \geqs 6$ (recall that $\l(A_4) = \l(F_4) = \l(E_6) = 6$ and $\l(E_7)=7$). For example, if $M = A_8$ then either $M_1$ is parabolic and $\l(M_1) \geqs 5$, or $M_1 \in \{B_4, A_2A_2\}$ and $\l(M_1) = 5$. As usual, the bound $\l(M) \geqs 6$ is easily checked when $M$ is parabolic.

Finally, let us assume $G = E_8$ and $p=2$. There is an unrefinable chain
\[
E_8 > D_8 > B_4 > B_2B_2 > B_2 > A_1A_1 > A_1 > U_1T_1 > T_1 > 1
\]
and thus $\l(G) \leqs 9$. To establish equality, we need to show that $\l(M) \geqs 8$. If $M$ is reductive then \eqref{e:M} holds and we consider each possibility in turn. The bound 
is clear if $M = A_1E_7$, $A_4A_4$ or $G_2F_4$ since $\l(E_7) = \l(F_4)=8$ and $\l(A_4)=9$. Similarly, $\l(A_2E_6) \geqs 8$ since $\l(A_2) = \l(E_6)=6$. If $M = A_8$ and $M_1$ is reductive, then $M_1=A_2A_2$ is the only option and thus $\l(M_1) \geqs 1+\l(A_2) = 7$. It is easy to see that $\l(M_1) \geqs 7$ if $M_1$ is parabolic, so $\l(A_8) \geqs 8$ as required. Similarly, one checks that $\l(D_8) \geqs 8$. Finally, if $M$ is a maximal parabolic subgroup with Levi factor $L$, then Lemma \ref{l:ss} gives $\l(M) \geqs 2+\l(L')$ and it is straightforward to show that $\l(L') \geqs 6$. The result follows.
\end{proof}

\subsection{Proof  of Thereom \ref{algp}(ii)} Now assume $G$ is a simple classical algebraic group. The next result establishes the upper bound on $\l(G)$ in part (ii) of Theorem \ref{algp}.

\begin{thm}\label{clup}
If $G$ is a simple classical algebraic group of rank $r$, then
\[
\l(G) \leqs 2(\log_2r)^2+12.
\]
\end{thm}

We partition the proof into a sequence of lemmas, starting with the case where $G$ is a symplectic group. Note that $\l(B_r) = \l(C_r)$ when $p=2$.

\begin{lem}\label{l:cr}
Let $G=C_r$ and write $r = 2^{a_1}+\cdots+2^{a_k}$ with $a_1> a_2 > \cdots > a_k \geqs 0$. Then
\[
\l(G) \leqs 4k-1+2\sum_{i}a_i \leqs 2(\log_2 (r+1))(\log_2 r) + 5
\]
and the conclusion of Theorem \ref{clup} holds.
\end{lem}

\begin{proof}
In view of Lemma \ref{l:lowrank}, we may assume $r \geqs 5$. First observe that $a_1 \leqs \log_2r$ and $k \leqs \log_2(r+1)$, so
\begin{align*}
4k-1+2\sum_{i}a_i \leqs 4k - 1 + 2(k\log_2r - k(k-1)/2) & = 2k\log_2r-k^2+5k-1 \\
& \leqs 2(\log_2 (r+1))(\log_2 r) + 5.
\end{align*}
Therefore, it suffices to show that
\begin{equation}\label{e:cr}
\l(G) \leqs 4k-1+2\sum_{i}a_i.
\end{equation}
We proceed by induction on $k$. We will repeatedly use the fact that if $H$ is a symplectic group with natural module $V$, then the stabiliser in $H$ of any proper nondegenerate subspace of $V$ is a maximal connected subgroup of $G$.

First assume $k=1$, so $r=2^{a_1}$ and $a_1 \geqs 3$. Let $M$ be the stabiliser in $G$ of a nondegenerate $r$-space, so $M = C_{2^{a_1-1}}C_{2^{a_1-1}}$. Now $M$ has a diagonally embedded maximal subgroup of type $C_{2^{a_1-1}}$, so there is an unrefinable chain
\[
C_{2^{a_1}} > C_{2^{a_1-1}}C_{2^{a_1-1}} > C_{2^{a_1-1}}.
\]
By repeating this process, we can descend from $G$ to $C_1$ in $2a_1$ steps and thus
\[
\l(G) \leqs 2a_1+\l(C_1) = 2a_1+3.
\]
This establishes the bound in \eqref{e:cr} when $k=1$.

Now assume $k>1$. Let $M$ be the stabiliser in $G$ of a nondegenerate $2^{a_1}$-space, so $M = C_{2^{a_1}}C_{2^{a_2}+\cdots + 2^{a_k}}$. By Lemma \ref{add}, we have
\[
\l(M) \leqs \l(C_{2^{a_1}}) + \l(C_{2^{a_2}+\cdots + 2^{a_k}})
\]
so by induction we get
\begin{equation}\label{e:cr2}
\l(G) \leqs 1 + (3+2a_1) + \left(4(k-1)-1+2\sum_{i=2}^{k}a_i\right)
\end{equation}
and thus \eqref{e:cr} holds.
\end{proof}

\begin{lem}\label{l:ar}
If $G=A_r$, then $\l(G) \leqs 2(\log_2r)^2+4$.
\end{lem}

\begin{proof}
By Lemma \ref{l:lowrank}, we may assume that $r \geqs 5$. If $r$ is odd then $C_{(r+1)/2}$ is a maximal connected subgroup of $G$, so Lemma \ref{l:cr} implies that
\[
\l(G) \leqs 1 + \l(C_{(r+1)/2}) \leqs 1 + 2(\log_2 ((r+3)/2))(\log_2 ((r+1)/2)) + 5 \leqs 2(\log_2r)^2+4
\]
as required. Similarly, if $r$ is even then
\[
A_r > U_rA_{r-1}T_1 > A_{r-1}T_1 > A_{r-1} > C_{r/2}
\]
is an unrefinable chain and a further application of Lemma \ref{l:cr} yields
\[
\l(G) \leqs 4 + \l(C_{r/2}) \leqs 4 + 2(\log_2 (r/2+1))(\log_2 (r/2)) +5 \leqs 2(\log_2r)^2+4.
\]
The result follows.
\end{proof}

\begin{lem}\label{l:dr}
Suppose $G=D_r$, where $r \geqs 3$. Then $\l(G) \leqs 2(\log_2r)^2+11$.
\end{lem}

\begin{proof}
As usual, we may assume $r \geqs 5$. First assume $p=2$ and let $M = B_{r-1}$ be the stabiliser of a nonsingular $1$-space. Since $\l(B_{r-1}) = \l(C_{r-1})$ when $p=2$, Lemma \ref{l:cr} implies that
\[
\l(G) \leqs 1 + \l(M) \leqs 1 + 2(\log_2(r-1))(\log_2 r)+5 < 2(\log_2 r)^2 + 6.
\]
For the remainder, we may assume $p \ne 2$.

Suppose $r$ is even and write $r=2^{a_1}+\cdots + 2^{a_k}$, where $a_1>a_2> \cdots > a_k \geqs 1$. We claim that the upper bound on $\l(G)$ in \eqref{e:cr} holds, in which case
\[
\l(G) \leqs 2(\log_2(r+1))(\log_2 r)+5.
\]
To prove this, we use induction on $k$. Note that if $M$ is the connected component of the stabiliser in $G$ of a nondegenerate $\ell$-space of the natural module, with $1 \leqs \ell \leqs r$ and $\ell \ne 2$, then $M$ is a maximal connected subgroup of $G$ (if $\ell=2$ then $M = T_1D_{r-1}$ is the Levi factor of a parabolic subgroup of $G$).

First assume $k=1$, so $r \geqs 8$. We can construct an unrefinable chain
\[
D_r > D_{2^{a_1-1}}D_{2^{a_1-1}} > D_{2^{a_1-1}} > \cdots > D_8 > D_4D_4 > D_4
\]
of length $2a_1-4$, so $\l(G) \leqs 2a_1-4+\l(D_4) \leqs 2a_1+3$.

Now assume $k>1$ (and continue to assume $r$ is even). Then $M = D_{2^{a_1}}D_{2^{a_2}+\cdots+2^{a_k}}$ is a maximal closed connected subgroup of $G$, so
\[
\l(G) \leqs 1 + \l(M) \leqs 1 + \l(D_{2^{a_1}}) + \l(D_{2^{a_2}+\cdots+2^{a_k}})
\]
and by induction we deduce that \eqref{e:cr2} holds. The result follows.

Finally, let us assume $r \geqs 5$ is odd. If $r=5$ then $\l(G) \leqs 6$ since $B_2$ is a maximal subgroup of $G$, so we can assume $r \geqs 7$. Let $M$ be the connected component of the stabiliser in $G$ of a nondegenerate $6$-space. Then $M = D_3D_{r-3}$ is a maximal connected subgroup of $G$, so by the previous result for even rank, we get
\[
\l(G) \leqs 1 + \l(D_3D_{r-3}) \leqs 6 + \l(D_{r-3}) \leqs 6 + 2(\log_2 (r-2))(\log_2(r-3)) + 5
\]
and this yields $\l(G) \leqs 2(\log_2r)^2+11$ as required.
\end{proof}

The next lemma completes the proof of Theorem \ref{clup}.

\begin{lem}\label{l:br}
Suppose $G=B_r$, where $r \geqs 3$. Then $\l(G) \leqs 2(\log_2r)^2+12$.
\end{lem}

\begin{proof}
This is an immediate corollary of Lemma \ref{l:dr} since $D_r$ is a maximal connected subgroup of $G$.
\end{proof}

\subsection{Proof of Theorem \ref{algp}(iii)}

Let $G=Cl(V)$ be a classical algebraic group of rank $r$ over an algebraically closed field $K$ of characteristic $p>0$, with  natural module $V$.  If $M$ is a maximal connected subgroup of $G$, then by \cite[Theorem 1]{LS98}, one of the following holds:
\begin{itemize}\addtolength{\itemsep}{0.2\baselineskip}
\item[(i)] $M$ is the connected stabiliser of a subspace $U$ of $V$ that is either totally singular, nondegenerate, or a nonsingular 1-space (the latter only when $G$ is orthogonal and $p=2$);
\item[(ii)] $M$ is the connected stabiliser $Cl(U)\otimes Cl(W)$ of a tensor product decomposition $V = U\otimes W$;
\item[(iii)] $M \in \mathcal{S}(G)$, the collection of maximal connected simple subgroups of $G$ such that $V$ is a $p$-restricted irreducible $KM$-module.
\end{itemize}

\begin{lem}\label{rkbd}
Let $G$ be as above,  let $M \in \mathcal{S}(G)$ and suppose $M$ is of classical type.
Then ${\rm rank}(M) > \sqrt{\log_p r}$.
\end{lem}

\begin{proof}
Let $k={\rm rank}(M)$. Using Weyl's character formula, it is easy to see that the $p$-restricted irreducible $KM$-module of largest dimension is the Steinberg module, which has  dimension $p^N$, where $N$ is the number of positive roots in the root system of $M$. Since $N\leqs k^2$, it follows that $\dim V \leqs p^{k^2}$. The conclusion follows, as $r<\dim V$.
\end{proof}

Now we prove Theorem \ref{algp}(iii). As in the statement, define $e_1(p)=p$, and $e_{l+1}(p) = p^{e_l(p)^2}$ for $l>1$, and for $x \in \R$ set
\[
\psi_p(x) = {\rm  min}\left(l \, : \, e_l(p) \geqs x\right).
\]
Note that $e_l(p) = \sqrt{\log_p e_{l+1}(p)}$, and  for $x> p$ we have
\begin{equation}\label{psieq}
\psi_p(x) = 1+\psi_p\left(\sqrt{\log_p x}\right).
\end{equation}

\begin{thm}\label{t:lim}
If $G$ is a simple algebraic group of rank $r$ in characteristic $p>0$, then
$\l(G) \geqs \psi_p(r)$.
\end{thm}

\begin{proof}
The proof proceeds by induction on $r$. If $r<e_3(p) = p^{p^{2p^2}}$, then $\psi_p(r) \leqs 2$, so the conclusion holds.

Now assume that $r \geqs e_3(p)$. Then certainly $r>8$, so $G$ is classical. Choose a maximal connected subgroup $M$ of $G$ such that $\l(M)=\l(G)-1$. Then $M$ is as in one of the possibilities (i)-(iii) above, and in case (iii) we have
${\rm rank}(M) > \sqrt{\log_p r}$, by Lemma \ref{rkbd}. In cases (i) and (ii), $M$ has a simple quotient $Cl(U)$ with $\dim U \geqs \sqrt{\dim V}$. Hence in any case, there is a simple connected group $H$ of rank at least $\sqrt{\log_p r}$, such that
$\l(M)\geqs \l(H)$. By induction,
\[
\l(H) \geqs \psi_p(\sqrt{\log_p r}),
\]
and so by (\ref{psieq}) we have
\[
\l(G) = 1+\l(M) \geqs 1+\psi_p(\sqrt{\log_p r}) = \psi_p(r).
\]
This completes the proof by induction.
\end{proof}

\vs

The  proof of Theorem \ref{algp} is now complete.

\vs

We conclude with an example showing that the lower bound $\psi_p(r)$ in Theorem \ref{algp}(iii) is of roughly the correct  order of magnitude.

\begin{ex}\label{slow}
Fix a prime $p \geqs 5$ and consider the series of embeddings of odd dimensional orthogonal groups via their Steinberg modules:
\[
B_{r_0} < B_{r_1} < \cdots < B_{r_k},
\]
where $r_0=1$, $r_1=(p-1)/2$ and $r_{l+1} = (p^{r_l^2}-1)/2$ for $l \geqs 1$. By \cite{sei}, each term in this series is maximal in the next, so $\l(B_{r_k}) \leqs k+3$.
\end{ex}

\subsection{Proof of Theorem \ref{cd0}}\label{cor3}

We begin by classifying the simple algebraic groups $G$ with $\l(G)=l(G)$.

\begin{lem}\label{cd0simple}
The only simple algebraic group $G$ satisfying $\l(G)=l(G)$ is $G = A_1$.
\end{lem}

\begin{proof}
First observe that $\l(A_1)=l(A_1)=3$, by Lemma \ref{depth23}. Conversely, suppose $G$ is simple of rank $r>1$ and $\l(G)=l(G)$. We know that $l(G) = l(B)+r$ by  Corollary \ref{length}.
If $r\leqs 4$ or $G$ is exceptional, this contradicts Lemma \ref{l:lowrank} or Theorem \ref{algp:ex}. And if $r\geqs 5$, then Theorem \ref{algp}(ii) gives a contradiction.
\end{proof}

Now we prove Theorem \ref{cd0}. Let $G$ be a connected algebraic group over an algebraically closed field.
Suppose $\l(G)=l(G)$ and $G$ is insoluble.  Set $\bar{G} = G/R(G)$ and note that $\l(\bar{G}) = l(\bar{G})$. Write $\bar{G} = G_1 \cdots G_t$, where each $G_i$ is simple. Then $\l(G_i) = l(G_i)$ for each $i$.
By Lemma \ref{cd0simple}, this implies that $G_i \cong A_1$ for all $i$. Since $\l(A_1A_1)=4 < l(A_1A_1)=6$, we must have $t=1$, so $\bar{G} \cong A_1$, as in Theorem \ref{cd0}.

\subsection{Proof of Theorem \ref{t:cd}}\label{s:cd}

Let $G$ be an algebraic group and recall that ${\rm cd}(G) = l(G) - \l(G)$ is the chain difference of $G$.

\begin{lem}\label{l:cd1}
If $N$ is a connected normal subgroup of $G$, then
\[
{\rm cd}(G) \geqs {\rm cd}(N) + {\rm cd}(G/N).
\]
\end{lem}

\begin{proof}
By Lemma \ref{add} we have $l(G) = l(N) + l(G/N)$ and $\l(G) \leqs \l(N) + \l(G/N)$. The conclusion follows.
\end{proof}

The analogous result for finite groups is \cite[Lemma 1.3]{BWZ}.

We now state some immediate consequences of the above lemma.

\begin{cor}\label{c:cd1}
\mbox{ }
\begin{itemize}\addtolength{\itemsep}{0.2\baselineskip}
\item[{\rm (i)}] If $N$ is a connected normal subgroup of $G$, then ${\rm cd}(G/N) \leqs {\rm cd}(G)$.
\item[{\rm (ii)}] If $1 = G_t \lhd G_{t-1} \lhd \cdots  \lhd  G_1 \lhd G_0 = G$ is a chain of connected subgroups of $G$, then ${\rm cd}(G) \geqs \sum_{i} {\rm cd}(G_{i-1}/G_i)$.
\item[{\rm (iii)}] If $G = G_1 \times \cdots \times G_t$, where each $G_i$ is connected, then
${\rm cd}(G) \geqs \sum_{i} {\rm cd}(G_i)$.
\end{itemize}
\end{cor}

The next result bounds $\dim G$ in terms of ${\rm cd(G)}$ when $G$ is simple.

\begin{prop}\label{p:simple}
Let $G$ be a simple algebraic group in characteristic $p \geqs 0$. Then
\[
\dim G \leqs \left\{\begin{array}{ll}
2 {\rm cd}(G) + 3 & \mbox{if $p=0$} \\
2 {\rm cd}(G) + 40 & \mbox{for any $p$.}
\end{array}
\right.
\]
\end{prop}

\begin{proof} Let $r$ be the rank of $G$.
Using Corollary \ref{length} and its notation we obtain
\begin{equation}\label{cdeq}
{\rm cd}(G) = \dim B + r - \l(G).
\end{equation}
Suppose first that $p=0$, and let $c$ be the value of $\l(G)$ as in Theorem \ref{alg0}.
Then we have
\[
{\rm cd}(G) = \dim B + r - c \geqs \dim B - 2.
\]
Therefore
\[
\dim G \leqs 2 \dim B - 1 \leqs 2{\rm cd}(G) + 3,
\]
as required.

Suppose now that $p>0$. Applying \eqref{cdeq} and Theorem \ref{algp} we obtain
\[
\dim G = 2\dim B-r \leqs 2{\rm cd}(G)+\lfloor 24-3r+4(\log_2r)^2 \rfloor,
\]
and the right hand side is at most $2{\rm cd}(G)+40$.
\end{proof}

The next result is of a similar flavour, dealing with certain semisimple groups.

\begin{prop}\label{p:ss}
Let $G = S^k$ where $k \geqs 2$ and $S$ is a simple algebraic group in characteristic $p \geqs 0$.
Then
\[
\dim G \leqs \left\{\begin{array}{ll}
2 {\rm cd}(G) + 2 & \mbox{if $p=0$} \\
2 {\rm cd}(G) + 28 & \mbox{for any $p$.}
\end{array}
\right.
\]
\end{prop}

\begin{proof}
By considering a series of diagonal subgroups, we can construct an unrefinable chain
\[
S^k > S^{k-1} > \cdots > S
\]
of length $k-1$, so $\l(S^k) \leqs k-1+\l(S)$. This yields
\[
{\rm cd}(G) = k\cdot l(S) - \l(S^k) \geqs k\cdot l(S) - (k-1+\l(S)) = k(l(S)-1)-\l(S) + 1.
\]
Corollary \ref{length} shows that if $B$ is a Borel subgroup of $S$, and $r = {\rm rank}(S)$, then
${\rm cd}(G) \geqs k(\dim B + r-1) -\l(S) + 1$, and so
\[
k\dim B \leqs {\rm cd}(G) - k(r-1) + \l(S) - 1.
\]
Since $\dim G = k\dim S \leqs k(2\dim B-1)$, we see that
\[
\dim G \leqs 2({\rm cd}(G) -k(r-1) + \l(S) - 1) - k = 2{\rm cd}(G) + a,
\]
where $a = 2 \l(S) - 2k(r-1) - 2 -k = 2 (\l(S)-1) - k(2r-1)$.

If $p=0$ then it is easy to check using Theorem \ref{alg0} that $a \leqs 2$ in all cases, as required.

Now suppose $p>0$. Then Theorem \ref{algp} yields
\[
a \leqs \lfloor 2\left(2(\log_2r)^2 + 11\right) - k(2r-1) \rfloor \leqs \lfloor 4(\log_2r)^2 - 4r + 24 \rfloor,
\]
which is at most $28$.
\end{proof}

\begin{lem}\label{summ}
Let $S_1,\ldots,S_n$ be simple algebraic groups that are pairwise non-isomorphic. Then
$\sum_{i} \dim S_i \geqs n^2$.
\end{lem}

\begin{proof}
Let $d_i = \dim S_i$, and assume $d_1\leqs d_2\leqs \cdots$. For any $r$, the number of distinct types of simple algebraic groups of rank at most $r$ is at most $4r$. Hence ${\rm rank}(S_i) \geqs \frac{i}{4}$, and so
\[
\sum_{i=1}^n d_i \geqs \frac{1}{16}\sum_{i=1}^n i^2 > \frac{1}{48}n^3.
\]
This is greater than $n^2$ provided $n>48$. For $n\leqs 48$, the conclusion can readily be checked by computation.
\end{proof}

We are now ready to prove Theorem \ref{t:cd}. Let $G$ be a connected algebraic group.
If $G$ is soluble then the conclusion holds trivially, so suppose $G$ is insoluble.
Let $R(G)$ be the radical of $G$ and write 
\[
\bar{G} = G/R(G) = \prod_{i=1}^{n} S_i^{k_i},
\] 
where the $S_i$ are pairwise non-isomorphic simple algebraic groups. By Corollary \ref{c:cd1}, ${\rm cd}(G) \geqs \sum_{i} {\rm cd}(S_i^{k_i})$, and hence Propositions \ref{p:simple} and \ref{p:ss} imply
\begin{equation}\label{cdineq}
{\rm cd}(G) \geqs \frac{1}{2}\sum_{i=1}^n \left(k_i\dim S_i - 40\right) = \frac{1}{2}\dim \bar{G} -20n.
\end{equation}
Now Lemma \ref{summ} gives ${\rm cd}(G) \geqs \frac{1}{2}(n^2-40n)$, and it follows that
\[
n \leqs 20 + \sqrt{400+2{\rm cd}(G)}.
\]
Therefore by \eqref{cdineq},
\[
\dim \bar{G} \leqs 2{\rm cd}(G)+40n \leqs 2{\rm cd}(G)+40\sqrt{400+2{\rm cd}(G)}+800.
\]
This completes the proof of Theorem \ref{t:cd}.

\begin{rem}\label{r:cd}
Let $G$ be an algebraic group in characteristic $p \geqs 0$ and set $\bar{G}=G/R(G)$. For a simple group $G$, it is easy to see that ${\rm cd}(G)=1$ if and only if $G = A_2$ and $p=2$. In the general case, by arguing as in the proof of Theorem \ref{cd0}, one can show that  
${\rm cd}(G)=1$ only if $\bar{G} = A_1$, or $p = 2$ and $\bar{G} = A_2$. For example, if $G = UA_1$, a semidirect product where $U$ is the natural module for $A_1$, then $l(G)=5$ and $\l(G)=4$.
\end{rem}

\section{Chain ratios}\label{chrat}

In this final section we consider the \emph{chain ratio} ${\rm cr}(G)$ of an algebraic group $G$, which is defined by
\[
{\rm cr}(G) = l(G)/\l(G).
\]
First we show that if $G$ is simple, then its dimension is bounded in terms of
its chain ratio.

\begin{prop}\label{p:crsimple}
Let $G$ be a simple algebraic group in characteristic $p \geqs 0$.
Then
\begin{itemize}\addtolength{\itemsep}{0.2\baselineskip}
\item[{\rm (i)}] $\dim G < 12\,{\rm cr}(G)$ if $p=0$.
\item[{\rm (ii)}] $\dim G < (1+o(1)) \cdot {\rm cr}(G) \cdot {(\log_2 {\rm cr}(G))}^2$ if $p > 0$,
where $o(1) = o_{{\rm cr}(G)}(1)$.
\end{itemize}
\end{prop}

\begin{proof}
Set $d = \dim G$ and note that $l(G) > d/2$ by Theorem \ref{t:len}. Therefore
\[
d < 2 l(G) = 2 \l(G)  {\rm cr}(G).
\]
Note that this holds for any algebraic group.

Assuming $G$ is simple and $p=0$, we have $\l(G) \leqs 6$ by Theorem \ref{alg0},
proving part (i).

Now suppose $p>0$ and let $r$ be the rank of $G$. Then $\l(G) \leqs 2(\log_2 r)^2 + 12$
by Theorem \ref{algp}. Since $d>r^2$ we obtain $\l(G) < \frac{1}{2}(\log_2d)^2 + 12$, so
\[
d <  ((\log_2 d)^2 + 24) \, {\rm cr}(G).
\]
This easily implies the conclusion of part (ii).
\end{proof}

In contrast to this result, we shall exhibit a sequence of algebraic groups $G$ for which $\dim G/R(G)$ is not bounded above in terms of the chain ratio ${\rm cr}(G)$. To show this we need the following result.

\begin{lem}\label{skdepth}
If $S$ is a simple algebraic group and $k \in \N$, then $\l(S^k) \geqs k+2$.
\end{lem}

\begin{proof}
The proof goes by induction on $k$, the case $k=1$ being clear.

Suppose $k>1$,  write $G=S^k$ and let $\pi_i:G\to S$ be the projection to the $i$-th factor. Let $M$ be a maximal connected subgroup of $G$ with $\l(M)=\l(G)-1$. If $\pi_i(M)=M_i<S$ for some $i$, then $M = M_i \times S^{k-1}$, and so $\l(M) \geqs \l(S^{k-1}) \geqs k+1$ by induction, giving the conclusion.

Now assume $\pi_i(M)=S$ for all $i$. Then $M$ is a product of diagonal subgroups of various subsets of the simple factors of $S^k$, and maximality forces $M = {\rm diag}(S^2)\times S^{k-2}$, where ${\rm diag}(S^2)$ denotes a diagonal subgroup of $S^2$. Hence $M \cong S^{k-1}$ and the conclusion again follows by induction.
\end{proof}

Now, fix a simple algebraic group $S$ and let
$G = S^k$ for $k \geqs 1$. Since $l(G) = k\cdot l(S)$ and $\l(G) \geqs k+2$ by Lemma \ref{skdepth}, it follows that
\[
{\rm cr}(G) = l(G)/\l(G) < k\cdot l(S)/k = l(S).
\]
Letting $k$ tend to infinity, we see that ${\rm cr}(G)$ is bounded, while
$\dim G/R(G) = \dim G$ tends to infinity.

\end{document}